\documentclass{amsart}
\author{Umar Hayat}
\title{A note on the canonical divisor of quasi-homogeneous affine algebraic varieties}
\date{Wednesday, October 24, 2012}
\address{Umar Hayat, School of Mathematical Sciences, University College Dublin, Belfield, Dublin 4, Ireland}
\email{umarmaths@gmail.com}
\usepackage{amsmath, amssymb, graphicx, amsfonts, amsthm}
\usepackage[english]{babel}
\usepackage{amsmath,amsfonts,amssymb}
\usepackage{setspace}

\numberwithin{thmNumber}{section}

\DeclareMathOperator{\GL}{GL}

\DeclareMathOperator{\Gr}{Gr}
\DeclareMathOperator{\Sec}{Sec}
\DeclareMathOperator{\Sp}{Sp}

\newcommand{\IC}{\mathbb{C}}

\newcommand{\IT}{\mathbb{T}}
\newcommand{\IP}{\mathbb{P}}

\newcommand{\cO}{\mathcal{O}}
\newtheorem{theorem}{Theorem}[section]
\newtheorem{remark}[theorem]{Remark}
\begin{document}
\begin{abstract}
We give a necessary and sufficient condition for the canonical divisor to vanish on a quasi-homogeneous affine algebraic variety.

\end{abstract}



\maketitle

\pagestyle{myheadings}

\markboth{UMAR HAYAT}{Quasi-homogeneous affine algebraic varieties}

\section{Introduction}
Let $G$ be a reductive algebraic group and $V$ a complex representation of $G$. Let $H\subset G$ be the stabiliser of a  vector $v\in V$. The variety $\Omega_{0}:=G/H=G \cdot v \subset \overline{G/H}=\Omega \subset V$ is a homogeneous space with a natural left $G$-action. We assume throughout that $\Omega_{0}$ in $\Omega $ has complement of codimension $\geq 2 $, or equivalently that the orbit $\Omega_{0}$ intersects every divisor (irreducible subvariety of codimension one) of $\Omega $. Under this assumption, we say $\Omega_{0}$ is a {\em big} open orbit.

Suppose  $\bigcup U_{i}$ is an open cover of $\Omega_{0}$. If $K_{\Omega_{0}} = 0$ then a global generator $\sigma \in \cO(K)$ can be written on each $U_{i}$ as 
\[
\sigma=(d\xi_{1}\wedge \dots \wedge d\xi_{n})/f_{i}
\]
where $\xi_{i}$ are the coordinates and $f_{i}$ is an invertible function on $U_{i}$. We can use this representation in calculating examples, although it does not play any role in the general statement and proof of our main result Theorem \ref{s!theo}.

\section{Vector bundles on $G/H$}
We can think of the quotient space $G/H$ as the base space of a bundle. The group $G$ is the total space for this bundle with a standard projection map $\pi \colon G \longrightarrow G/H$ that takes $g$ to the coset $\overline{g}=gH$ . Each coset $gH \in G/H$ is fixed by $gHg^{-1}$ under the transitive action of $G$ on $G/H$. For each $\overline{g}$, we denote  the fibre over $\overline{g}$ by $\pi^{-1}(\overline{g})$, which is a copy of $H$. We can think of the total space of the bundle as a family of copies of $H$, parametrised by points of the base space. This bundle is called a principal $H$-bundle.  See \cite{Eqi,FH,Hum} and \cite[Chapter 9]{Naka} for the construction of vector bundles and line bundles on $G/H$ and properties of their Lie algebras.

Let $\rho$ be any representation of $H$ on a vector space $E$. We construct a $G$-equivariant vector bundle over $G/H$ with fibre $E$ as follows. We quotient out the product space $G\times E$ by the following action of $H$:
\[
(g,e)\in (G,E)\rightarrow (gh, \rho(h^{-1})e). 
\] 
We denote this quotient space by $G\times_{H} E$. It is a vector bundle because the fibre over each point $\overline{g}\in G/H$ is a copy of the vector space $E$. In general we can think of a vector bundle as a family of vector spaces of same dimension, parametrised by the base space. In particular each one-dimensional representation of $H$ gives us a line bundle.

To summarise the discussion above, let $\mathcal{C}$ be the category of $G$-equivariant vector bundles on $G/H$ and $\mathcal{D}$ the category of  representations of $H$. 
For a given functor $F\colon \mathcal{C}\rightarrow \mathcal{D}$ and an equivariant vector bundle $E$ over $G/H$ we get a representation of $H$ on the fibre $[H]$, where $[H]$ is the identity coset. On the other hand, given a functor $G\colon \mathcal{D}\rightarrow \mathcal{C}$, we can get a vector bundle over $G/H$ for a given finite dimensional  representation $E$ of $H$.

We denote Lie algebras of $G$ and $H$ by $\mathfrak{g}$ and $\mathfrak{h}$ respectively. The tangent bundle $T_{G/H}$ to $G/H$ comes from the representation $\mathfrak{g/h}$, where $\mathfrak{g/h}$ is the tangent space to $G/H$ at the identity $H$. The tangent space to any other $gH\in G/H$ is given by $\mathfrak{g}/g\mathfrak{h}g^{-1}$. The top wedge of the tangent bundle $\overset {\mathrm {top}}{\bigwedge} T_{G/H}$  gives rise to the anticanonical line bundle. It can be expressed as the product of those weights of $G$ which are not weights of $H$. The canonical class $K_{G/H}$ of the variety $G/H$ comes from the dual $\overset{\mathrm { top}}{\bigwedge} T^{\vee}_{G/H}$.

\begin{remark}
The tangent space to $G/H$ at $[H]$ is $\mathfrak{g}/\mathfrak{h}$. This corresponds to the $H$-module $\mathfrak{g}/\mathfrak{h}$ where $H$ acts on $\mathfrak{g}/\mathfrak{h}$ by the adjoint action. The line bundle corresponds to $M = \overset{\mathrm { top}}{\bigwedge} T^{\vee}_{G/H}$ on which $H$ acts as a character $\kappa \colon H\rightarrow \IC^{\ast}$. We consider the line bundle $K_{X}$ as a $G$-equivariant sheaf that corresponds to $\kappa$. 

\end{remark}

\section{Statement of the main result}\label{s!conj}
Let $G$ be a connected reductive algebraic group, $T$ a maximal torus of $G$ such that $\IT_{H}=T\cap H$ be the maximal torus of $H$ where $H$ is a closed connected subgroup of $G$.

Now we are ready to state and prove our main result.

\begin{theorem}\label{s!theo}
The canonical class $K_{G/H}$ of the homogeneous space $G/H$ is trivial if and only if the action of the restricted torus $\IT_{H}$ on $\mathfrak{g/h}$ has trivial determinant.
\end{theorem} 

\begin{proof}
The first observation is that the category $\mathcal{C}$ of $G$-equivariant vector bundles on $G/H$ and the category $\mathcal{D}$ of representation of $H$ are equivalent; see for example \cite{Eqi}, page $1$.

Since $M$ has rank  $1$, it is trivial on the commutator $[H, H]$ and actually
a representation of the abelianisation $H_{a} = H/[H, H]$. Now, every connected abelian affine algebraic group is isomorphic to the direct product of $G_{u} \times G_{s}$ where $G_{u}$ is a connected group whose elements are all unipotent and $G_{s}$ is a connected group whose elements are all semisimple, see \cite [Chapter 1]{Sha}. Also since $G_{u}$ is unipotent so it has no nontrivial characters. Furthermore $G_{s}$ is isomorphic to a direct product $\IC^{\ast} \times \cdots \times \IC^{\ast}$ of groups isomorphic to the multiplicative group $\IC^{\ast}$. In this case $G_{u}$ is isomorphic to a product $\IC^{+} \times \cdots \times \IC^{+}$ of groups isomorphic to the additive group $\IC^{+}$. Therefore the abelian group $H_{a}$ is isomorphic to a product $V \times T_{1}$ of a vector space and a torus
$T_{1}$, where $V\cong \IC^{+} \times \cdots \times \IC^{+}$ and $T_{1}= \IC^{\ast} \times \cdots \times \IC^{\ast}$.

If $\IT_{H} = T \cap H$ is the maximal torus of $H$ for some choice of $T$ then the natural map $\IT_{H} \rightarrow T_{1}$ is surjective; see \cite[p.136]{Hum} and result follows because we want $M$ as an $H$ module on which $H$ acts as a character but $\IT_{H} \rightarrow T_{1}$ is surjection so it is enough to take the action of $\IT_{H}$ on $M$ which acts by determinant.

\end{proof}

\section{Examples and further applications}

It is well known that a  scheme is called Gorenstein if it is Cohen-Macaulay and the canonical divisor $K$ is Cartier. Theorem \ref{s!theo} provides an alternative way in terms of representation theory  to check one part of a variety being Gorenstein. We give two examples that illustrate Theorem \ref{s!theo}, by satisfying its hypothesis and being Gorenstein. 
\subsection{Secant variety of the Grassmannian $\Gr(2,n)$} 
We denote by \newline $\Omega= \Sec^{2}( \Gr(2,n))$ the second secant variety of the Grassmannian $\Gr(2,5)$ in its Pl\"{u}cker embedding, where $n\geq 5$. The affine cone $\Omega_{0}\subset \bigwedge^{2}\IC^{n}$ is the orbit of $v=e_{1}\wedge f_{1}+e_{2} \wedge f_{2}$ under the action of $G=\GL(n,\IC)$. The subgroup $H$  of $G$ that stabilises the $v$ is given by
\[
 H =
 \begin{pmatrix}
 A_{4\times 4} & *_{4\times n-4} \\ 
 0_{(n-4)\times 4} & *_{(n-4)\times (n-4)}\\
 \end{pmatrix},
  \]
where $A\in \Sp(4,\IC)$ and (*) means there are no restrictions on these entries.
We know that the Lie algebras of $G$ and $\Sp(4, \IC)$ have dimension $n^{2}$ and $10$, respectively, therefore $\mathfrak{g/h}$ has dimension $4n-10$.
The restricted torus $\IT_{H}$ is given by
\[
 \begin{pmatrix}
  t_{1} & 0 & \cdots & 0 \\
  0 & t_{2} & \cdots & 0 \\
  \vdots  & \vdots  & \ddots & \vdots \\
  0 & 0 & \cdots & t_{n} \\
  \end{pmatrix}
    \]
with $t_{1}t_{2}t_{3}t_{4}=1$.

Here we calculate the canonical class using action of $\IT_{H}$ on $\mathfrak{g}/\mathfrak{h}$, when we take the action of $\IT_{H}$  on $\mathfrak{g/h}$ there are a total of $4n-10$ weight spaces for $\IT_{H}$. There are six weight spaces with weight $1$. The product of the remaining weights is given by
\[
\dfrac{ (t_{5} \cdots t_{n})^{4} }{(t_{1} \cdots t_{4})^{n-4}} 
.
\]
We know that in case of the restricted torus $t_{1}t_{2}t_{3}t_{4}=1$, hence the canonical class $K_{G/H}$ of the second secant variety  is a multiple of ($t_{5} \dots t_{n})^{4}$ which is the determinant of the restricted torus $\IT_{H}$. This verifies the determinant condition from theorem \ref{s!theo}.
\subsection{Rational normal curve}
Suppose $G=\GL(2,\IC)$ acts on $\IC^{2}$. We are consider orbit of the highest weight vector $e_{1}$.
In this case $\IP^{1}=G/P=\IP(G\cdot e_{1})$, where 
\[
P=
\begin{pmatrix}
  \ast & \ast  \\
  0 & \ast  \\
  
  \end{pmatrix}
  = 
  \{   
  \begin{pmatrix}
  a_{11} & a_{12}  \\
  0 & a_{22}  \  
  
  \end{pmatrix}\\ \vert a_{11}a_{22} \neq 0 
  \}.
  \] 

Let $\chi_{k}: P\rightarrow \IC^{\times}$ be the character for $P$ given by $a_{11}^{k}$ and the kernel that corresponds to this character is given by
\[
H_{k}=
\begin{pmatrix}
  \mu_{k} & \ast  \\
  0 & \ast  \\
  
  \end{pmatrix}, 
\] 
where $\mu_{k}$ is a primitive $k$th root of unity.
We can give coordinates on one affine piece by $\mathfrak{g}/\mathfrak{h}_{k}$. We want to calculate the  product of the weights that are only weights for $G$ but not of $H_{k}$ under the action of the torus $\IT_{H}\subset G$ on $\mathfrak{g}/\mathfrak{h}_{k}$, where $\mathfrak{g}$ and $\mathfrak{h}_{k}$ denote the Lie algebras of $G$ and $H_{k}$ respectively.

The maximal torus 
\[
\IT_{H}=
\begin{pmatrix}
  t_{1} & 0 \\
  0 & _{t_{2}} \\
  
  \end{pmatrix} 
\]
acts on 
\[
\begin{pmatrix}
  0 & 0 \\
  1 & 0 \\
  
  \end{pmatrix} 
\] 
by  
\[
\IT_{H}
\begin{pmatrix}
  0 & 0 \\
  1 & 0 \\
  
  \end{pmatrix} 
  \IT_{H}^{-1}
\]
 and we get  
\[
\begin{pmatrix}
  0 & 0 \\
  1 & 0 \\
  
  \end{pmatrix} 
 \] 
  as weight vector with weight $\dfrac{t_{2}}{t_{1}}$, where \[
\begin{pmatrix}
  0 & 0 \\
  1 & 0 \\
  
  \end{pmatrix} 
 \]  is the weight vector for $G$ and not for $H$. As we know that $\chi_{k}$ is the character for $P$, when we restrict this to the maximal torus then this character is given by $t_{1}^{k}$. We can see that  the canonical differential is a multiple of the trivial determinant of the $\IT_{H}$ exactly when $k=1 \text{ or  } 2$.

\bigskip 
 
{\bf Acknowledgements:} I wish to thank Dmitriy Rumynin for his valuable suggestions to prove the main result. This work was mainly carried out during my  PhD studies at the University of Warwick under the supervision of Miles Reid, to whom I would like to express my deepest gratitude. I acknowledge useful scientific discussions with Stavros Papadakis and Stephen Coughlan.

\bibliographystyle{alpha}
\bibliography{paper}

\end{document}